\documentclass[oneside]{amsart}% --- Packages ---
\UseRawInputEncoding
\usepackage[utf8]{inputenc}   % for pdflatex
\usepackage[T1]{fontenc}      % ensures accented characters render correctly
\usepackage{geometry}                % Page geometry
\usepackage{graphicx}                % For figures
\usepackage{amssymb, amsmath}        % AMS math
\usepackage{enumitem}                % Custom lists
\usepackage{hyperref}                % Clickable references
\usepackage{epstopdf}                % EPS to PDF conversion

\newtheorem{theorem}{Theorem}[section]
\newtheorem{lemma}[theorem]{Lemma}
\newtheorem{comment}[theorem]{Comment}

\newtheorem{proposition}[theorem]{Proposition}
\newtheorem{conjecture}[theorem]{Conjecture}
\newtheorem{definition}[theorem]{Definition}
\newtheorem{example}[theorem]{Example}

\newtheorem*{theorem*}{Theorem}   % Unnumbered version

\title{Restricted Marstrand's projection theorem for general families of linear subspaces}
\author{Jiahan Du}
\address{University of California, Los Angeles}
\email{jiahandu@math.ucla.edu}
\date{\today}

\begin{document}

\begin{abstract}
This paper investigates a refinement of Marstrand's projection theorem; more specifically, let $\Pi_t, t\in[0,1]$ be a family of $m$ dimensional subspaces of the Euclidean space $\mathbb{R}^n$ and let $P_t:\mathbb{R}^4\mapsto \Pi_t$ be the orthogonal projections onto $\Pi_t$. We hope to determine the conditions on $\Pi_t$ under which, for any Borel $A\subset\mathbb{R}^n$, $\dim_H P_t(A)=\min(m,\dim_H A)$ holds for almost every $t$. We propose a conjectured condition on $\Pi_t$ and provide partial progress towards its resolution. We first establish a version of the polynomial Wolff axiom, and then apply polynomial partitioning to derive a version of the $L^p$ Kakeya inequality. Finally, we use a discretization procedure to obtain the desired bound.
\end{abstract}
\maketitle

\section{Introduction}
In 1954, Marstrand \cite{Marstrand} proved the following theorem, starting a long line of research on projection theorems.

\begin{theorem}[Marstrand \cite{Marstrand}]
Let $A\subset\mathbb{R}^2$ be a Borel measurable set, and $P_t:\mathbb{R}^2\rightarrow\mathbb{R},t\in[0,2\pi]$ be the orthogonal projection onto the span of $(\cos t,\sin t)$. Then $\dim_H P_t(A)=\min(\dim_H A,1)$ for almost every $t\in[0,2\pi]$. 
\end{theorem}

Here and in the future, we use $\dim_H$ to denote Hausdorff dimension in Euclidean spaces. Marstrand's original proof is elementary but somewhat complex. Later, Kauffman \cite{Kaufman} provided a much simpler proof of Marstrand's result using potential theory. His method also generalizes to higher dimensions:

\begin{theorem}[Kaufman\cite{Kaufman}]\label{Kaufman}
    Let $A\subset\mathbb{R}^n$ be a Borel measurable set and $1\leq m\leq n$, if $K\in G(n,m)$, and we denote the orthogonal projection from $\mathbb{R}^n$ onto $K$ to be $P_K$, then $\dim_H P_K(A)=\min(\dim_H A,m)$ for almost every $K\in G(n,m)$. 
\end{theorem}

Here $G(n,m)$ is the Grassmannian of $m$ dimensional subspaces of $\mathbb{R}^n$. Since then, researchers have discovered various generalizations and refinements of Marstrand's original result, and connections to other problems in mathematics, e.g. the sum product theorem by Bourgain \cite{Bourgain}.

Later, Gan, Guo, Guth, Harris, Maldague and Wang in \cite{plane} introduced new tools from harmonic analysis into the study of this topic. They used the so called decoupling inequalities to prove the following refinement of theorem \ref{Kaufman}. Using similar methods, Gan, Guo and Wang proved a further generalization in \cite{projection}:

\begin{theorem}[Restricted projection theorem \cite{projection}]\label{projection} Let $\gamma:[0,1]\rightarrow\mathbb{R}^n$ be a smooth curve such that $\det(\gamma(t),...,\gamma^{(n-1)}(t))\not=0$ for all $t\in[0,1]$ and $1\leq m\leq n-1$. Let $P_t$ be the orthogonal projection onto $\text{span}(\gamma(t),...,\gamma^m(t))$ and let $A\subset\mathbb{R}^n$ be Borel, then $\dim_H P_t(A)=\min(\dim_H A,m)$ for almost every $t\in[0,1]$. 
    
\end{theorem}

Their method actually proves a stronger statement. Their method shows that if $A\subset\mathbb{R}^n$ is Borel measurable and its Hausdorff dimension is greater than $a$ for some $a>0$. Let $s\in[0,m)$, define $E=\{t\in[0,1]|\,\dim_H(P_t(A))<s\}$, then $\dim_H E\leq 1+\frac{s-a}{n-m}$. However we will not be concerned with stronger estimates like this here.

The proof of theorem \ref{projection} uses the following result in harmonic analysis: 

\begin{theorem}[Decoupling inequality for the moment curve \cite{moment}]
Let $\gamma:[0,1]\rightarrow\mathbb{R}^n$ be the smooth curve given by $\gamma(t)=(t,t^2,...,t^n)$ and $1\leq k\leq n$. Let $\delta>0$ and $P(\delta)$ be a partition of $[0,1]$ into disjoint intervals of length $\delta$. For each interval $I$, we let $c_I$ be the center of $I$, let $U_I$ be the parallelpiped of dimensions $|I|\times...\times|I|^k$ whose center if $\gamma(c_I)$ and sides parallel to $\partial^1(\gamma)(c_I),...,\partial^k(\gamma)(c_I)$. Then if for each $I$, $f_I$ is a smooth function with Fourier support in $U_I$, then for $p\in [2,k(k+1)]$, then for any $\epsilon>0$, there exists a constant $C_\epsilon$ independent of the $f_I$'s, such that $$\Big|\Big|\sum_{I\in P(\delta)} f_I\Big|\Big|_{L^p}\leq C_\epsilon\delta^{-\epsilon}\bigg(\sum_{I\in P(\delta)}||f_I||_{L^p}^2\bigg)^{1/2}$$
\end{theorem}
\noindent We comment that a similar result holds for any smooth curve $\gamma:[0,1]\rightarrow\mathbb{R}^n$ such that $$\det|\gamma^{(1)}(t),...,\gamma^{(n)}(t)|>0$$ for all $t.$

In summary, Gan, Guo, Guth, Harris, Maldague and Wang first converted the original problem about Hausdorff dimension into trying to prove an $L^p$ inequality about sum of indicator of parallelpipeds, then they divided up the support of the Fourier transform of those indicator functions into different regions. Each region can either be estimated directly, or be estimated using the decoupling inequality. Due to the form of the decoupling inequality from above, Gan, Guo and Wang only considered families of subspaces of the form $\text{span}(\gamma,...,\gamma^{(k)})$.

In this paper, we ask the following question, given a family of linear maps $P_t:\mathbb{R}^n\rightarrow\mathbb{R}^m, t\in[0,1]$, $n\geq m$, when is it true that if $A\subset\mathbb{R}^n$ is Borel, we have $\dim_H(P_t(A))=\min(\dim_H A,m)$ for almost every $t\in[0,1]$? We make the following conjecture, which essentially says if the restricted Marstrand's theorem does not hold, then there must be counterexamples among subspaces:

\begin{conjecture}\label{1c}
    Suppose $P_t:\mathbb{R}^n\rightarrow\mathbb{R}^m$ is an analytic family of linear maps, in the sense that its coordinate functions are analytic functions of $t$, and assume that for any subspace $\pi$ of $\mathbb{R}^n$, $\dim P_t(\pi)=\min(\dim\pi,m)$ for all but finitely many $t\in[0,1]$. We have $\dim_H P_t(A)=\min(\dim_H A,m)$ for almost every $t$ whenever $A\subset\mathbb{R}^n$ is Borel.
\end{conjecture}

A more general version of the above conjecture is:

\begin{conjecture}\label{2c}
Suppose $P_t:\mathbb{R}^n\rightarrow\mathbb{R}^m$ is an analytic family of linear maps. Let $f:[0,n]\rightarrow\mathbb{R}$ be defined as follows. If $x\in[0,n]$ is an integer, $f(x)$ is the largest integer such that for any subspace $A$ of $\mathbb{R}^n$ with $\dim A=x$, $\dim P_t(A)\geq f(x)$ for all but finitely many $t\in[0,1]$. If $x$ is not an integer, $f$ is defined to be linear on the interval $[\lfloor x \rfloor, \lfloor x \rfloor+1]$. We conjecture that if $A\subset\mathbb{R}^n$ is Borel, then $\dim_H P_t(A)\geq f(\dim_H A)$ for almost every $t$.
\end{conjecture}

For example, consider the linear maps $P_t:\mathbb{R}^{2n}\mapsto \mathbb{R}^n$ given by $P_t((x_1,y_1,...,x_n,y_n))=(x_1+ty_1,...,x_n+ty_n)$, then we can prove that if $A$ is Borel, $\dim_H(P_t(A))\geq f(\dim_H A)$ for almost every $t$. Here $f(x)=\lfloor \frac{x+1}{2}\rfloor$ if $x$ is an integer, and $f$ is linear between integers. We will be focusing on the less general conjecture \ref{1c} in this paper.

We give the following partial progress, which is the main result of the paper:

\begin{theorem}\label{Main}
Suppose $\Pi_t, t\in[0,1]$ is a family of two dimensional subspaces of $\mathbb{R}^4$ and let $P_t:\mathbb{R}^4\rightarrow\Pi_t$ be the orthogonal projection onto $\Pi_t$. Furthermore assume that we have $b_1(t), b_2(t): [0,1]\rightarrow\mathbb{R}^4$ such that $\Pi_t=\text{span }(b_1(t),b_2(t))$ and each coordinate function of $b_i$ is a polynomial in $t$. Furthermore, assume that for any subspace $\pi$ of $\mathbb{R}^4$, $\dim P_t(\pi)=\min(\dim\pi,2)$ for all but finitely many $t\in[0,1]$. Then there exists a constant $\epsilon>0$ depending only on $P_t$, such that $\dim_H(P_t(A))>1+\epsilon$ for almost all $t$ whenever $A$ is a Borel subset of $\mathbb{R}^4$ and $\dim_H A=2$. Here, because of the inefficiency of the argument, $\epsilon$ is a positive constant that is small if the degrees of the coordinate functions of $b_1(t)$ and $b_2(t)$ are large.

\end{theorem}
\begin{comment}
It is not hard to see that orthogonal projection onto $\mathrm{span}(b_1(t),b_2(t))$ and the linear map $\mathbb{R}^4\rightarrow\mathbb{R}^2$ given by $x\mapsto (x\cdot b_1(t),x\cdot b_2(t))$ are releated by a fixed linear transformation. Therefore, the Hausdorff dimensions of a Borel subset $A\subset\mathbb{R}^4$ under the orthogonal projection $P_t$ and under $x\mapsto(x\cdot b_1(t),x\cdot b_2(t))$ are identical.
\end{comment}
\begin{example}
We can take $P_t$ to be the orthogonal projection onto the span of $[1,0,t,t^2]$ and $[0,1,t^2,-t^3]$. A direct computation shows that this case satisfies the assumptions of theorem \ref{Main} but is not covered by theorem \ref{projection}. Later we show see that for every Borel $A\subset\mathbb{R}^4$, we have $\dim_H P_t(A)\geq 1+\frac{1}{5}$ for almost every $t\in[0,1].$
\end{example}
\begin{lemma} \label{reduction}
Under the assumptions of theorem \ref{Main}, suppose $\Pi_t=\text{span }(\gamma_1(t),\gamma_2(t))$ such that all coordinate functions of $\gamma_i$ are analytic functions. We may assume that  
\[
\det([\gamma_1(t), \gamma_2(t), \gamma_1'(t), \gamma_2'(t)]) \neq 0
\]
for all $t \in [0,1]$.

\end{lemma}
\begin{proof}
The determinant $f(t)=\det(\gamma_1(t),\gamma_2(t),\gamma_1'(t),\gamma_2'(t))$ is an analytic function of $t$, hence it either has finitely many zeros or is constant zero on $[0,1]$. In the first case by redefining the domain of $t$, we may assume $f(t)\not=0$ for all $t\in[0,1]$ by restricting to subintervals of $[0,1]$. In the second case there must exist analytic functions $a(t),b(t),c(t),d(t): [0,1]\rightarrow\mathbb{R}$ of $t$ such that $$a\gamma_1+b\gamma_2=c\gamma_1'+d\gamma_2'$$ and $a(t),b(t),c(t),d(t)$ must not be simultaneously zero for any $t\in[0,1]$. Let $\gamma=c\gamma_1+d\gamma_2$, then $$\text{span}(\gamma,\gamma')=\text{span}(c\gamma_1+d\gamma_2,a\gamma_1+b\gamma_2+c'\gamma_1+d'\gamma_2)$$$$\\=\text{span}(c\gamma_1+d\gamma_2,(a+c')\gamma_1+(b+d')\gamma_2).$$

The function $c(b+d')-d(a+c')$, being an analytic function of $t$ has finitely many zeros or is the constant zero function. In the first case $\text{span}(\gamma,\gamma')=P_t$ for all but finitely many $t\in[0,1]$,  and the conjecture has been reduced to a special case already proven by theorem \ref{projection}. In the second case we have $c(b+d')=d(a+c')$. If $c$ is the constant zero function, then $ad=0$. But $d(t)$ must be nonzero for any $t\in[0,1]$ for otherwise we have $a(t)\gamma_1(t)+b(t)\gamma_2(t)=0$, hence $a(t)=b(t)=c(t)=d(t)=0$ and a contradiction. Hence $a$ is the constant zero function.  This implies $b\gamma_2=d\gamma_2'$, from which we can solve $\gamma_2(t)=\vec{c}\cdot \exp(\frac{b}{d}t)$ for some constant vector $\vec{c}\in\mathbb{R}^4$. The projection of the three dimensional subspace perpendicular to $\vec{c}$ onto $P_t$ is then at most one dimensional for all $t$. This contradicts the assumption. Therefore $c$ is not the constant zero function, so it only has finitely many zeros. Similarly $d$ only has finitely many zeros. By restricting to subintervals of $[0,1]$ if necessary, we can find an analytic function $\lambda$ of $t$
 such that $a+c'=\lambda c$, $b+d'=\lambda d$, hence $$(\lambda c-c')\gamma_1+(\lambda d-d')\gamma_2=c\gamma_1'+d\gamma_2',$$ hence $\lambda (c\gamma_1+d\gamma_2)=(c\gamma_1+d\gamma_2)'$, then $c\gamma_1+d\gamma_2$ is a constant vector times $e^{t\lambda}$, in this case the projection of the three dimensional subspace perpendicular to this constant vector to $P_t$ is at most one dimensional for all $t$. This contradicts the assumption. The above discussion then shows that we may assume $\det(\gamma_1(t),\gamma_2(t),\gamma_1'(t),\gamma_2'(t) )\not=0$ for any $t\in[0,1]$. 
\end{proof}

We are unable to prove the full conjecture above due to inefficiencies in the polynomial partitioning argument at the moment when this was written. However, the recent resolution of three dimensional Kakeya conjecture by Hong Wang and Joshua Zahl in \cite{proofkakeya} might be able to be adapted to our case, provided that their argument for straight tubes can be generalized to curved tubes. However, it is nontrivial to generalize their method to curved tubes, and it might involve making their arguments work for semialgebraic sets instead of convex sets. 

We devote the rest of the paper to proving theorem \ref{Main}. The method we use was first invented in \cite{2021} and used again in \cite{zahl}. To briefly illustrate the idea, let $P_t:\mathbb{R}^n\rightarrow\mathbb{R}^m$ be a family of linear maps and $A\subset\mathbb{R}^n$ be Borel measurable, then consider the collection of curves in $\mathbb{R}^{m+1}$ given by $c_a:[0,1]\mapsto\mathbb{R}^{m+1}, c_a(t)=(t,P_t(a))$ for each $a\in A$, then roughly speaking to prove that $\dim_H P_t(A)=\min(m,\dim_H A)$ for almost every $t$, we only need to show that $\dim_H \bigcup_{a\in A}c_a=\min(m+1,\dim_A+1)$. This problem has the same form as the curved variant of the Kakeya problem, studied for example in \cite{Wisewell}. In \cite{2021} and \cite{zahl}, the authors studied projections from $\mathbb{R}^3$ to $\mathbb{R}$, which correspond to unions of curves in $\mathbb{R}^2$. Because of the low dimension (dimension 2), the authors were able to study tangencies directly.

In theorem \ref{Main}, however, we need to deal with curves that lie in $\mathbb{R}^3$, and we use the polynomial partitioning method first invented by Larry Guth to make progress on the restriction conjecture. We will first adapt an argument by Nets Katz and Keith Rogers in \cite{Katz} to prove a version of the polynomial Wolff axiom, and use the polynomial Partitioning method to convert it to the needed curved Kakeya bound. Similar curved Kakeya problems have also been studied recently in \cite{guo}, where the authors applied the polynomial Wolff axioms for curves to make partial progress on a conjecture involving $L^p$ bounds of certain Hormander operators.

For the rest of the proof, in section 2 we  prove a version of polynomial Wolff axiom, in section 3 we use the polynomial Wolff axiom to derive a multilinear Kakeya inequality, in section 4 we use the well known "broad narrow" method to convert the multilinear Kakeya inequality to a linear Kakeya quality, and in the last section we discuss how to go from linear Kakeya inequality to restricted problem. 

Before we end the introduction, we see a non-example. We take $P_t:t\in[0,1]$ as the orthogonal projection onto the span of $\big((1,0,-2t^2,-2t),(0,1,-2t,0)\big)$. Then the corresponding unions of curve problem is the following: 

Let $A\subset\mathbb{R}^4$ be Borel, for each $x=(x_1,x_2,x_3,x_4)\in A$, we define a curve $\gamma_x:[0,1]\rightarrow\mathbb{R}^3$ by $$\gamma_x(t)=(t,x_1-2t^2x_3-2tx_4,x_2-2tx_3),$$ is it true that the Hausdorff dimension of $\bigcup_{x\in A}\gamma_x$ is $\dim_H A+1$?

The answer is false and a counterexample is provided by Laura Wisewell in \cite{Wisewell}: we can simply pick $A$ to be the subspace given by $x_1=0$ and $x_2=-2x_4$; all curves $\gamma_x$ would then lie on the surface $x=yz$ in $\mathbb{R}^3$. Alternative, we can check the restricted projection problem also fails with $A=\text{span}\big((1,0,0,0),(0,0,0,1)\big)$ as expected.

\textbf{Acknowledgment:} The author is deeply grateful to Hong Wang, who introduced me to this topic and taught me much of what I know. I also want to thank Arian Nadjimzadah for pointing out Laura Wisewell's thesis on the curved Kakeya problem to me.

\vspace{1em}
\textbf{Notations:} In the following paper, we shall use $C>0$ to denote a large constant that depends only on $P_t$, and we use $C(X,Y,...)$ to denote a large constant that depends only on quantities $X,Y,...$. We use $X\lesssim Y$ to denote $X\leq CY$, and we use $X\lessapprox Y$ to denote $X\lesssim (\log \delta^{-1})^C Y.$

\section{Polynomial Wolff Axioms for Curved Tubes}
Our goal this section is to establish a version of the polynomial Wolff axiom for curved tubes. We start with the following lemma:

\begin{lemma}\label{integral}
Let $P_t:\mathbb{R}^n\rightarrow\mathbb{R}^m, t\in[0,1]$ be a family of linear maps given by $$P_t(x,v)=(x\cdot \beta_1(t),...,x\cdot\beta_m(t))+(v\cdot\gamma_1(t),...,v\cdot\gamma_m(t)),$$ where $n=2m$, $x\in\mathbb{R}^m$, $v\in\mathbb{R}^{m}$ and each $\beta_i(t), \gamma_i(t):[0,1]\rightarrow\mathbb{R}^{m}$ is a vector of polynomial functions in $t$. Furthermore, we assume that for every linear subspace $A\subset\mathbb{R}^n$, we have $\dim P_t(A)=\min(m,\dim A)$ for all but finitely many $t\in[0,1].$ Denote the map $$(x,v,t)\mapsto (x\cdot \beta_1(t),...,x\cdot\beta_m(t))+(v\cdot\gamma_1(t),...,v\cdot\gamma_m(t))$$ by $\Phi(x,v,t):\mathbb{R}^{n+1}\mapsto\mathbb{R}^m$. Given a new positively oriented orthonormal basis $e_1,...,e_n$ of $\mathbb{R}^n$, there is a constant $C>0$ such that for every $\alpha\in[0,1]$, $\lambda\in[0,1-\alpha]$ and $m\times m$ matrix $M$, we have

$$\int_\alpha^{\alpha+\lambda}\big|\det(\nabla_{e_1,...,e_{m}}\Phi\cdot M+\nabla_{e_{m+1},...,e_n}\Phi) \big|\,dt\geq C\lambda^N.$$ Here $N\in\mathbb{Z}^+$ depends on the degrees of coordinate functions of $\beta_i$'s and $\gamma_i$'s only, and we adopt the convention that for a function $f:\mathbb{R}^n\rightarrow\mathbb{R}^m$, $\nabla_{e_1,...,e_k}f$ is the $m\times k$ matrix given by $(\nabla_{e_1,...,e_k}f)_{ij}=\nabla_{e_j}f_i$. 
\end{lemma}
\begin{proof}
Let $A$ denote the orthogonal matrix $[e_1,...,e_n]^{-1}$ and we write $A=\begin{bmatrix}
    A_1 & A_2\\
    A_3 & A_4
\end{bmatrix}$, where each $A_i$ is a square matrix of dimension $m\times m$. Define the square matrix $\beta$ by $\beta:=\left[\begin{smallmatrix}
    \beta_1\\
    \beta_2\\
    ...\\
    \beta_m
\end{smallmatrix}\right]$ and $\gamma$ by $\gamma:=\left[\begin{smallmatrix}
    \gamma_1\\
    \gamma_2\\
    ...\\
    \gamma_m
\end{smallmatrix}\right]$; we can calculate by change of variable:

$$\int_\alpha^{\alpha+\lambda}\big|\det(\nabla_{e_1,...,e_{m}}\Phi\cdot M+\nabla_{e_{m+1},...,e_n}\Phi) \big|\,dt=\int_\alpha^{\alpha+\lambda} \big|\det\big(\beta(A_1M+A_2)+\gamma(A_3M+A_4) \big)\big|\,dt$$$$=\int_\alpha^{\alpha+\gamma}\Big|\det\Big([\beta\,\,\gamma]\cdot\left[\begin{smallmatrix}
    A_1M+A_2\\
    A_3M+A_4
\end{smallmatrix}\right]\Big)\Big|\,dt.$$ Denote the $n\times m$ matrix $\left[\begin{smallmatrix}
    A_1M+A_2\\ A_3M+A_4
\end{smallmatrix}\right]$ by $K$, then the integrand above is a polynomial in $t$, with coefficients as polynomials of entries of $K$. Assume $$\det\Big([\beta\,\,\gamma]\cdot\left[\begin{smallmatrix}
    A_1M+A_2\\
    A_3M+A_4
\end{smallmatrix}\right]\Big)=\sum_{h=0}^B p_h(K_{ij})t^h,$$ where $B$ is the degree of the above polynomial in $t$ depending on $\gamma_i$'s only, each $p_h(y_{ij})$ is a polynomial in $nm$ variables $y_{ij}: 1\leq i\leq n, 1\leq j\leq m$, and we use $p_h(K_{ij})$ to denote $p_h(y_{ij})$ evaluated at $K_{ij}$. By assumption, the polynomials $p_h$ satisfy the following condition: if for some choice of $y_{ij},1\leq i\leq n, 1\leq j\leq m$, we have $p_h(y_{ij})$ is zero for each $h=1,...,B$, then the $n\times m$ matrix $[y_{ij}]$ has rank strictly less than $m$. 

Next, We note a simple equality about determinants: if $A_1,...,A_m$ are a list of $n$ by $n$ matrices, then \begin{equation}\label{determinant}
\det(\sum_i A_i)=\sum_f\det K_f
\end{equation} where the summation is over all functions $f:\{1,...,,n\} \mapsto \{1,...,m\}$; $K_f$ is the matrix whose $i$th row is the $i$th row of $A_{f(i)}$. This equality can be derived using multilinearity of determinants with respect to rows as follows: denote the $j$th row of $A_i$ by $A^{(j)}_i$, then 

$$\det(\sum_i A_i)=\det\Bigg(\begin{bmatrix}
    \sum_i A^{(1)}_i \\ \sum_i A^{(2)}_i  \\ ... \\ \sum_i A^{(n)}_i
\end{bmatrix}\Bigg)=\sum_i \det\Bigg(\begin{bmatrix}
    A^{(1)}_i \\ \sum_i A^{(2)}_i  \\ ... \\ \sum_i A^{(n)}_i
\end{bmatrix}\Bigg)$$
In the second equality above we expanded the determinant with respect to the first row, taking advantage of multilinearity. Equation \ref{determinant} follows by performing the same expansion across all other rows. 

We can rewrite equation \ref{determinant} as follows: 

\begin{equation}
    \det(\sum_i A_i)=\sum_f c(f)\det A_f,
\end{equation}
where the summation is over all functions $f:\{1,...,n\}\rightarrow\{1,...,m\}\times\{1,...,n\}$. Here $c(f)$ is a constant that depends only on $f$ (actually it can be taken to be either one or zero).  $A_f$ is the $n\times n$ matrix whose $i$th row is $f(i)(2)$th row of $A_{f(i)(1)}$.

If instead we have four square matrices $A,X,B,Y$ of dimension $n\times n$, then can first expand $XB$ into sum of multiple matrices as follows: $$\det(XA+YB)=\det(\sum_{i,j}X_{ij}I_{ij}A+\sum_{i,j}Y_{ij}I_{ij}B),$$ here $I_{ij}$ is the $n$ by $n$ matrix that is $1$ for the $(i,j)$ element but zero elsewhere.  We can then apply equation \ref{determinant} to this new sum. After discarding matrices that have at least one row with only zeros, we can conclude $$\det(XA+YB)=\sum_g p_g(X_{ij},Y_{ij})\det K_g,$$ where the summation is over all functions $g: \{1,...,n\}\mapsto \{1,2\}\times \{1,...,n\}$. $K_g$ is the $n\times n$ matrix whose $i$-th row is the $g(i)(2)$th row of $A$ if $g(i)(1)=1$, and the $g(i)(2)$th row of $B$ if $g(i)(1)=2$. $p_g$ is a polynomial over matrix elements of $X$ whose form depends on $g$ only. For example, in the case $n=2$, let $
A=\left[\begin{smallmatrix} a_1\\ a_2 \end{smallmatrix}\right]$, $
B=\left[\begin{smallmatrix} b_1\\ b_2 \end{smallmatrix}\right]$ where $a_1, a_2, b_1, b_2$ are all row vectors of length 2, then we can calculate:

\[
\begin{aligned}
\det(XA+YB)
&=(X_{11}X_{22}-X_{12}X_{21})\det\!\left[\begin{smallmatrix} a_1\\ a_2 \end{smallmatrix}\right]
 +(X_{11}Y_{21}-X_{21}Y_{11})\det\!\left[\begin{smallmatrix} a_1\\ b_1 \end{smallmatrix}\right]\\
 &+(X_{11}Y_{22}-X_{21}Y_{12})\det\!\left[\begin{smallmatrix} a_1\\ b_2 \end{smallmatrix}\right]
 +(X_{12}Y_{22}-X_{22}Y_{12})\det\!\left[\begin{smallmatrix} a_2\\ b_2 \end{smallmatrix}\right] \\
&\qquad
 +(Y_{11}Y_{22}-Y_{12}Y_{21})\det\!\left[\begin{smallmatrix} b_1\\ b_2 \end{smallmatrix}\right]
 +(X_{12}Y_{21}-X_{22}Y_{11})\det\!\left[\begin{smallmatrix} a_2\\ b_1 \end{smallmatrix}\right].
\end{aligned}
\] 

For any $n\times m$ matrix $Q$, we denote by $Q_I\in\mathbb{R}^{\binom{n}{m}}$ the vector of $m\times m$ minors. Apply the matrix equality above to $\det\big(\beta(A_1M+A_2)+\gamma(A_3M+A_4)\big)$, we see that each polynomial $p_h$ can be rewritten to be $c_h\cdot K_I$ for some constant vector $c_h\in\mathbb{R}^{\binom{n}{m}}$. Therefore we can rewrite 
$$\det\big(\beta(A_1M+A_2)+\gamma(A_3M+A_4)\big)=\sum_{h=0}^B (c_h\cdot K_I) t^h.$$
So far we have rewritten $p_h$ as a polynomial of the $m\times m$ minors $K_I$'s of $K$. Moreover each $p_h$ is a linear combination of $K_I$'s. 

We now record an observation in \cite{Katz}:
\begin{lemma}[\cite{Katz}]\label{lower}
Given some $n\in\mathbb{Z}^+$, there is a constant $C>0$ such that given a polynomial $P(t)$ of the form $P(t)=t^n+\text{lower degree terms}$, we have $\int_0^1|P(t)|\geq C$. 
\end{lemma}
\begin{proof}
Let $z_1,...,z_n$ be the (possibly complex) roots of $P(t)$, then $P(t)=(t-z_1)...(t-z_n)$, then the set $\{t\in[0,1], |t-z_i|<\frac{1}{4n}\text{ for some }i\}$ has Lebesgue measure at least $1/2$. 
\end{proof}

We proceed with proof by contradiction. Assuming our lemma is false for any positive $N$.  By choosing the power $N$ in lemma \ref{key} sufficiently large and using lemma \ref{lower}, we have a sequence of $K^{(i)}, i=1,2,...$ such that $c_h\cdot K^{(i)}_I:=\epsilon_h\in\mathbb{R}\rightarrow 0$ for each $h$, and we would like to show that this implies that each coordinate of $K^{(i)}_I$ goes to zero as $i\rightarrow\infty$. Let $S$ be the image of the Plucker's embedding, which we may identify with a subset of the unit sphere of $\mathbb{R}^{\binom{n}{m}}$, then $S$ is a compact subset of the unit sphere. Furthermore, we know by assumption there does not exist an $x\in S$ such that $c_h\cdot x=0$ for all $h$, for otherwise we can find an $n\times m$ matrix $K$ such that $[\beta, \gamma]\cdot K=0$. The $m$ dimensional subspace $A$ spanned by the column vectors of $K$ then violates the assumption of the lemma. Since $S$ is compact, each $p_h$ is a continuous function on the compact set $S$ that is nowhere zero. This implies that there exists $\epsilon>0$ such that $||c_h\cdot x||>\epsilon$ for all $h$ and any $x\in S$. However we have $||c_h\cdot K^{(i)}_I||\rightarrow 0$ as $i\rightarrow\infty$, which implies that $||K^{(i)}_I||\rightarrow 0$ as $i\rightarrow\infty$. 

However, it is impossible for all $m\times m$ minors of $K^{(i)}$ to converge to zero as $i\rightarrow\infty$. For each column vector $(x_1,...,x_n)$, we can identify it with the differential form $x_1dx_1+....+x_ndx_n$. Then the wedge product of all column vectors of $
\left[ \begin{smallmatrix}
A_1 \\ A_3
\end{smallmatrix} \right]$ and all column vectors of $K$ is the wedge product of all column vectors of $A$. This is a contradiction as the coefficients of $A$ are constants, and do not tend to zero as $i\rightarrow \infty$. 
\end{proof}

We shall need the case $n=4, m=2$ of the lemma above.  We now use the above integral inequality to show a version of the polynomial Wolff axiom for curved tubes as in \cite{Katz} and later \cite{guo}. Follow the setup of the previous lemma, we define $\Phi(x,v,t): \mathbb{R}^2\times\mathbb{R}^2\times\mathbb{R}\rightarrow\mathbb{R}^2$ by $$\Phi(x,v,t)=\begin{bmatrix}
    x\cdot \beta_1(t) \\
    x\cdot \beta_2(t)
\end{bmatrix}+\begin{bmatrix}
    v\cdot \gamma_1(t) \\
    v\cdot \gamma_2(t)
\end{bmatrix}$$ given a Borel set $A\subset\mathbb{R}^4$, we have a set of curves given by $t\mapsto (\Phi(x,v,t),t)$ for $(x,v)\in A$. Later in section 5 we will have the following situation: we have a collection $A'$ of points $(x,v)$ and their associated curves $t\mapsto (\Phi(x,v,t),t)$, unfortunately the directions $v$ are not necessarily $\delta$ disjoint. However there exists a two dimensional subspace $\Pi\subset\mathbb{R}^4$, such that the image of $A'$ under the orthogonal projection $P:\mathbb{R}^4\mapsto\Pi$ is $\delta$ separated. Let $(e_1,e_2,e_3,e_4)$ be another standard basis of $\mathbb{R}^4$, such that $e_3,e_4$ span $\Pi$ and $e_1,e_2$ span $\Pi^\perp$, then $\Phi(x,v,t)$ can be rewritten using a change of variable into $e_1,e_2,e_3,e_4$ and $t$.  We will rename $e_1,e_2,e_3,e_4$ to be $(x,v)$ for notational simplicity, where $x=(e_1,e_2), v=(e_3,e_4)$. For $\delta>0,\alpha>0$ and $\lambda\in[0,1-\alpha]$, we define the following curved $\delta$ tube:

\begin{equation}\label{curvedtube}
T_{x,v,[\alpha,\alpha+\lambda],\delta}=\{(x',t)\in\mathbb{R}^3\,|\,|x'-\Phi(x,v,t)|\leq\delta, x'\in\mathbb{R}^2, t\in[\alpha, \alpha+\lambda]\}
\end{equation}
and we say a $\lambda\times\delta$ tube like this has position $x$, direction $v$. The curved tubes are defined using the new coordinate system, and in this new coordinate system of course $\Phi$ has the same form $\Phi(x,v,t)=\begin{bmatrix}
    x\cdot \beta_1(t) \\
    x\cdot \beta_2(t)
\end{bmatrix}+\begin{bmatrix}
    v\cdot \gamma_1(t) \\
    v\cdot \gamma_2(t)
\end{bmatrix}$ as above, but with different $\beta_i(t)$ and $\gamma_i(t)$. 

The point of this change of variables is as follows: for each $\lambda\times\delta$ tube we associate a "label" $(x,v)\in\mathbb{R}^4$ to it, and we (essentially arbitrarily) assign $v$ to be the direction. Later we shall show that separation of the "direction" $v$ in the "label" implies good separation property on the tubes. 

 \begin{proposition}\label{firstone}
Let $E$ be a positive integer and $\epsilon>0$, then there is a constant $C(E,\epsilon)>0$ so that for every set $\mathbb{T}$ of $\lambda\times\delta$ tubes in $[0,1]^3$, pointing in $\delta$ separated directions, we have $$\#\{T\in\mathbb{T}: T\subset S\}\leq C(E,\epsilon)|S|\delta^{-2-\epsilon}\lambda^{-N}$$ whenever $S$ is a semialgebraic set of complexity at most $E$ and $\lambda\geq\delta>0$.  Note the loss $\lambda^{-N}$ is worse than the loss for straight Kakeya tubes. Here $N$ is the same large integer as in lemma \ref{integral}.
\end{proposition}
The method we use is invented by Nets Katz and Keith Rogers in \cite{Katz}. Their method has been used and refined in for example \cite{guo} and \cite{oldzahl}. Before we start the proof, we first record some basic facts about semialgebraci sets and a few results from \cite{Katz}.

\begin{definition}[Complexity]
    Let $E\subset\mathbb{R}^n$ be a semialgebraic set given by finite unions of sets of the form $\{(x_1,...,x_n)\in\mathbb{R}^n|\,P(x_1,...,x_n)=0\}$ and $\{(x_1,...,x_n)\in\mathbb{R}^n|\,Q(x_1,....,x_n)>0\}$  where $P,Q$ are multivariate polynomials. The complexity of $E$ is the minimum of the sum of degrees of all defining polynomials. 
\end{definition}

\begin{definition} [Dimension]
    Let $S\subset\mathbb{R}^n$ be a semialgebraic set, the dimension of $S$ is the largest integer $d$ such that there exists an injective semialgebraic map from $[0,1]^d$ into $S$. 
\end{definition}

The following Tarski's projection theorem says the collection of semialgebraic sets forms what logicians call an o-minimal structure.

\begin{theorem}[Tarski] For any positive integers $n\geq 2$ and $E$, there is a constant $C(n,E)$ such that for every semialgebraic set $S\subset\mathbb{R}^n$ of complexity at most $E$, the projection of $S$ onto the first $n-1$ coordinates has complexity at most $C(n,E)$. 
\end{theorem}

From Tarski's projection theorem, \cite{Katz} showed the the following lemma for straight tubes. The proof carries over to our curved tubes.

\begin{lemma}[curved case of \cite{Katz}]\label{katz1}
Let $S\subset\mathbb{R}^3$ be a semialgebraic set of complexity at most $E$ and let $\alpha\in\mathbb{R}$, $L_s=\{(x,v)\in[0,1]^4,T_{x,v,[\alpha,\alpha+\lambda],\delta}\subset S\}$ is a semialgebraic set of complexity at most $C(E)$, where $C(E)$ is a constant depending on $E$ only. 
\end{lemma}

The following lemma shows that you can always pick a semialgebraic slice.

\begin{lemma}[\cite{Katz}]\label{semi_section}Let $n\in\mathbb{Z}^+$ and $S\subset\mathbb{R}^{2n}$ be a compact semialgebraic set of complexity at most $E$, then there is a semialgebraic set $Z\subset\mathbb{R}^{2n}$ of complexity at most $C(n,E)$ such that $\Pi(Z)=\Pi(S)$ and $\Pi$ is a bijection on $Z$, where $\Pi$ is the projection of $\mathbb{R}^{2n}$ onto its first $n$ coordinates and $C(n,E)$ is a constant depending on $n$ and $E$ only. 
\end{lemma}

We also need the following lemma, first proved by Yomdin and subsequently generalized by Gromov:

\begin{lemma}[Gromov] For any semialgebraic set $A\subset[0,1]^n$ of dimension $d$ and complexity at most $E$ and positive integer $r$, we can find maps $\phi_1,..,\phi_N: [0,1]^d\rightarrow[0,1]^n$ so that  $$\bigcup_{i=1}^N \phi_i([0,1]^d)=A\text{           
 and           }||\phi_j||_{C^r}=\max_{|\alpha|\leq r}||\partial^\alpha\phi_j||_\infty\leq 1,$$ here $N$ is bounded by a constant depending on $n,r,E$ only.
\end{lemma}

We also need a version of the Bezout's theorem from theorem 5.2 of \cite{chen2011}:
\begin{theorem}[Bezout's theorem] Let $f_1,...,f_n$ be $n$ polynomials with coefficients in $\mathbb{R}$ over $n$ variables $x_1,...,x_n$, a zero $X=(x_1^\ast,...,x_n^\ast)$ is an called an isolated root if $\det(\nabla f)(X)\not=0$. If the total degree of $f_i$ is $d_i$, then the number of isolated zeros to the system of equations $f_i(x_1,...,x_n)=0, i=1,...,n$ is bounded above by $d_1\cdot...\cdot d_n$.
\end{theorem}

Now we go back to the proof of proposition \ref{firstone}. 

\begin{proof}
 We first cover $[0,1]$ with intervals $I_1,...,I_n$ of length $\lambda/2$ such that different $I_i,I_j$ overlap if and only if $i+1=j$. The projection of $T\in\mathbb{T}$ into the third coordinate must contain some $I_k$. For each $k$, we let $\mathbb{T}_k$ be the collection of curved tubes in $\mathbb{T}$ whose projection onto the third coordinate contains $I_k$. We let $S_k$ be the subset of $S$ consisting of points whose projection on to the third coordinate is a in $I_{k-1}\cup I_k\cup I_{k+1}$. Then it suffices to prove $$\#\{T\in\mathbb{T}_k: T\subset S_k\}\leq C(E,\epsilon)|S_k|\delta^{-2-\epsilon}\lambda^{-N}$$ for each $k$.
Clearly, we can assume $|S_k|\geq\lambda\delta^2$, for otherwise $S_k$ cannot contain a single curved tube. We now use proof by contradiction: suppose for each constant $C^\ast>0$, we can find a collection of $\lambda\times\delta$ tubes $\mathbb{T}$ with $\delta$ separated directions such that $$\#\{T\in\mathbb{T}: T\subset S_k\}> C^\ast|S_k|\delta^{-2-\epsilon}\lambda^{-N}.$$ Choose a big constant $C>0$, depending on $p,q,r$ and $s$. Let $\alpha$ be the left end point of $I_k$  and consider the set $$L=\{(x,v)\in[0,1]^4: T_{x,v,[\alpha,\alpha+\lambda],\frac{\delta}{C}}\subset S_k\},$$ this set is a semialgebraic set by lemma \ref{katz1}. Furthermore, we can check by direct computation that $T_{x,v',[\alpha,\alpha+\lambda],\frac{\delta}{C}}\subset T_{x,v,[\alpha,\alpha+\lambda],\delta}$ if $|v'-v|<=\frac{\delta}{C}$. Next we can use lemma \ref{semi_section} to obtain a semialgebraic section $L'\subset L$ consisting of a single $(x,v)$ for each $v$ appearing in $L$. Let $\Pi$ denote the projection $(x,v)\rightarrow v$, then 
$|\Pi(L')|=|\Pi(L)|>C^\ast|S_k|\delta^{-\epsilon}\lambda^{-N}$ by our proof by contradiction hypothesis, where the constants $C^\ast$ might be different from each other by a factor depending only on $\epsilon$ and $E$. The semialgebraic dimension of $L'$ is 2, because the projection onto $v$ restricted to $L'$ is a semialgebraic bijection to a two dimensional set. By Gromov's algebraic lemma, we can break $L'$ into $N$ pieces such that for each piece $L_j$, there is a map $(F_j,G_j):[0,1]^2\rightarrow[0,1]^{4}$ with $(F_j,G_j)([0,1]^2)=L_j$ and $||(F_j,g_j)||_{C^r}\leq 1$, where we take $r=\lfloor\frac{C^2}{\epsilon}\rfloor$ and $N$ is bounded above by a number depending on $\epsilon$ and $E$ only. The pigeonhole principle then implies that there exists $j$ such that $$|G_j([0,1]^2)|=|\Pi(L_j)|> C^\ast|S_k|\delta^{-\epsilon}\lambda^{-N}.$$ This implies that we have a ball $B\subset[0,1]^2$ centered at some $x_0$ and of radius $\delta^{\frac{\epsilon}{4}}$ so that $$|G_j(B)|>C^\ast|S_k|\delta^{-\frac{\epsilon}{2}}\lambda^{-N}.$$ Now replace $(F_j,G_j)$ by $(F,G)$ the $(r-1)$th Taylor approximation. Note that the length of the boundary of $G(B)$ is bounded above by $O(\delta^{\frac{\epsilon}{4}})$ and recall we assume $|S_k|>\delta\lambda^2$, this implies that $|G(B)|>C^\ast|S_k|\lambda^{-N}$ at least if $C$ is large enough since we have $|(F_j,G_j)(x)-(F,G)(x)|\leq |x-x_0|^r\leq\delta^{\frac{C^2}{4}}$. Moreover, since the tube $T_{F_j(x),G_j(x),[\alpha,\alpha+\lambda],\frac{\delta}{C}}$ is contained in $S_k$ for $x\in B$ by assumption, we see that  $(\Phi((F(x),G(x),t),t)\in S_k$ if $x\in B$ and $t\in I$. This implies 
\begin{equation}
    |S_k|\geq\int_I|\Phi(F(x),G(x),t)(B)|\,dt. 
\end{equation}
We would like to deal with $\Phi(F(x),G(x),t)(B)$ by a change of variable in $x$. However $x\mapsto \Phi(F(x),G(x),t)$ might not be a diffeomorphism onto its image as it can be a many to one function and may not even be a local diffeomorphism. To remedy this, let $B_t\subset B$ be the subset of $x$ where $\nabla_x \Phi(F(x),G(x),t)$ is invertible.  We have $|\Phi(F(x),G(x),t)(B)|=|\Phi(F(x),G(x),t)(B_t)|$ by Sard's theorem. Furthermore, by Bezout's theorem, for each $t$, $\Phi(F(x),G(x),t)$ is at most an $(r-1)^2$ to one function on $B_t$. Therefore, we conclude that:

\begin{equation*}
\begin{split}
|S_k|\geq\int_I|\Phi(F(x),G(x),t)(B)|\,dt &\geq\frac{1}{(r-1)^{2}}\int_I\int_{B_t}|\det\nabla_x\Phi(F(x),G(x),t)|\,dx\,dt\\
&=\frac{1}{(r-1)^{2}}\int_I\int_{B} |\det\nabla_x\Phi(F(x),G(x),t)|\,dx\,dt\\
&=\frac{1}{(r-1)^{2}}\int_I\int_{B}|\det (\nabla_x\Phi\cdot\nabla F+\nabla_v\Phi\cdot\nabla G)|\,dx\,dt
\end{split}
\end{equation*}
 We can now use lemma \ref{integral} to bound the RHS to be $$\geq\frac{1}{(r-1)^{2}}\int_{B}\big|\det \nabla G\big|\int_{I}|\det(\nabla_x\Phi\cdot\nabla F\cdot\nabla G^{-1}+\nabla_v\Phi)|\,dt\,dx$$

$$\gtrsim \frac{1}{(r-1)^{2}}\int_{B}|\det\nabla G|\lambda^{N}\,dx\geq\frac{1}{(r-1)^{2}}G(B)\lambda^N>\frac{C^\ast|S_k|}{(r-1)^{2}},$$ note in the first inequality, on the set where $\det\nabla G(x)=0$ for $x\in B$ the inequality is true trivially as the integrand on the right is 0, and the third inequality $\int_B|\det\nabla G|\geq G(B)$ is true by the change of variable formula. Now we simply set $C^\ast>0$ to be large enough we have a contradiction. This completes the proof. 
 \end{proof}

 The version of polynomial Wolff axiom we need is slightly more general. The proof for the straight case as in \cite{Katz} carries over, except for some technicalities which can be checked.

The following is the next step towards the needed polynomial Wolff axiom. If $A$ is a subset of a Euclidean space and $\delta>0$ is small, then we use $A_\delta$ to denote the $\delta$ neighborhood of $A$. 

 \begin{lemma}\label{weakha}
 Let $S\subset\mathbb{R}^3$ be a semialgebraic set of complexity at most $E$, and let $\epsilon>0$, then there is a constant $C(E,\epsilon)$, such that for every collection $\mathbb{T}$ of tubes of shape $1\times\delta$, we have  $$\#(\{T\in\mathbb{T}:|T\cap S|\geq\lambda|T|\})\leq C(E,\epsilon)|S_{\delta}|\delta^{-2-\epsilon}\lambda^{-N}.$$  Here $N$ is the same large integer as in lemma \ref{integral}.
 \end{lemma}
\begin{proof}
 Let $\Pi$ be projection onto the third coordinate. Note that if $|T\cap S|\geq\lambda|T|$, then the Fubini theorem implies that there is a parallel translate $l$ of the central curve $t\mapsto \Phi(x,v,t)$ of $T$ such that $l$ is contained in $T$ and $|l\cap S|\geq\lambda$. The fundamental theorem of algebra then implies that $l\cap S$ has at most $C(E)$ connected components for a constant $C(E)$ depending on $E$ only. Therefore there is a connected component of length at least $C(E)^{-1}\lambda$; hence $T_\delta\cap S_\delta$ contains a $C(E)^{-1}\lambda\times\delta$ curved tube with the same central curve as $T$, apply the previous proposition then concludes the proof. 
\end{proof}

Next, we replace $|S_\delta|$ by $|S|$ and obtain the polynomial Wolff axiom we need. Before that, we record a result due to Wongkew:
\begin{theorem}[Wongkew \cite{Wongkew}]
Let $m$ be the codimension of a real algebraic variety whose defining polynomials are all bounded in degree by $d$. And let $B$ be an arbitrary $n$-ball in $\mathbb{R}^n$ with radius $R$. There exist constants $c_i$ which depend only on $n$, so that for all positive $\rho$ the following is true: $$\text{vol }((V\cap B)_\rho)\leq\sum_{j=m}^n c_jd^j\rho^j R^{n-j}.$$
\end{theorem}

For the next proposition, we need a technical condition on $e_1,e_2,e_3$ and $e_4$. Assume $$\Phi(x,v,t)=\begin{bmatrix}
    (x,v)\cdot (\beta^\ast_1,\gamma^\ast_1)\\ (x,v)\cdot (\beta^\ast_2,\gamma^\ast_2)
\end{bmatrix}$$ in this new coordinates. Of course each coordinate function of $\beta^\ast_i, \gamma^\ast_i$ is a linear combination of coordinate functions of $\beta_i$'s and $\gamma_i$'s. We require that $\gamma_1^\ast(t), \gamma_2^\ast(t)$ are independent for each $t$. This holds true for almost every choice of the basis vectors $(e_i)$. This condition is necessary to ensure that given a maximally $\delta$ separated set $V$ of $v$ and given a fixed $x$, the set $\{\Phi(x,v,t)|\,v\in V\}$ is also maximally $\sim\delta$ separated.

\begin{proposition}\label{weak}
Given $E\in\mathbb{Z}^+$ and $\epsilon>0$, there is a constant $C(E,\epsilon)>0$ depending on $E$ and $\epsilon$ only such that for every collection $\mathbb{T}$ of $1\times\delta$ tubes with $\delta$ separated directions and any semialgebrac set $E\subset\mathbb{R}^3$ of complexity at most $E$, we have $$\#\{ T\in\mathbb{T}: |T\cap S|\geq\lambda |T|\} \leq C(E,\epsilon) |S|\delta^{-2-\epsilon}\lambda^{-N}.$$
\end{proposition}
\begin{proof}
It would be nice if there is an upper bound of $|S_\delta|$ in terms of $|S|$. First we may assume without loss of generality that $|S|\geq\lambda\delta^2$. Let $\eta>0$ be arbitrary, we have $|S_\eta|\leq |S|+|(\partial S)_\eta|$ where $\partial S$ is the boundary of $S$. By the Milnor-Thom theorem \ref{Milnor}, $(\partial S)_\eta$ is contained in the $\eta$ neighborhood of at most $C(E)$ hypersurfaces of degree at most $E$. Wongkew's result cited above then shows that $|(\partial S)_\eta|\leq C(E)\eta$. Since $|S|\geq\delta^3$, if we take $\eta=\delta^3$, we then have $|S_\eta|\leq C(E)|S|$.

Therefore, it suffices to show that under the given assumptions in proposition \ref{weak}, we can upgrade lemma \ref{weakha} to: 
\begin{equation}\label{key}\#(\{T\in\mathbb{T}:|T\cap S|\geq\lambda|T|\})\leq C(E,\epsilon)|S_{\eta}|\delta^{-2-\epsilon}\lambda^{-N}.\end{equation} 

Here we need to make an elementary observation: we define a bush to be a collection of $1\times\eta$ curved tubes that have $\eta$ separated directions. Then there exists an absolute constant $K$, such that for any given $1\times\delta$ tube $T_{x,v,[a,a+1],\delta}$, we can find $\leq K$ number of bushes that together cover $T_{x,v,[a,a+1],\delta}$, with each bush containing $\leq(\delta/\eta)^2$ number of curved $1\times\eta$ tubes. Furthermore, the direction of each $\eta$ tube is within $\delta$ of the direction of the thicker tube $T_{x,v,[a,a+1],\delta}$. To see why, note that the collection of tubes $T_{x,v,[a,a+1],\eta}$ for $\eta$ in an $\eta$ separated subset of $B(v,\delta)$ together covers $T_{x,v,[a+1-1/C,a+1],\delta/C}$ for some absolute constant $C$ large enough. $T_{x,v,[a,a+1],\delta}$ can be covered by $\lesssim 1$ number of such radius $\frac{\delta}{C}$ tubes, with different $x$'s but the same $v$.

Denote $c=1/K$. Let $\mathbb{T}_\lambda=\{T\in\mathbb{T}:|T\cap S|\geq\lambda|T|\}$, and for each $T\in\mathbb{T}_\lambda$, we can find a bush $\mathbb{Y}_T$ such that $\sum_{Y\in\mathbb{Y}_T}|Y\cap S|\geq c\lambda\delta^2$. Let $\mathbb{Y}=\bigcup_{T\in\mathbb{T}_\lambda}\mathbb{Y}_T$, and we partition $\mathbb{Y}$ into subsets $\mathbb{Y}_k=\{Y\in\mathbb{Y},2^{-k}\lambda|Y|\leq |Y\cap S|<2^{-k+1}\lambda|Y|\}$, $k\geq\log_2\lambda, k\in\mathbb{Z}$. Note that on the one hand, we sum to get \begin{equation}\label{1}\sum_{k\geq\log_2\lambda}\sum_{Y\in\mathbb{Y}_k}|Y\cap S|=\sum_{T\in\mathbb{T}_\lambda}\sum_{Y\in\mathbb{Y}_T}|Y\cap S|\geq c\lambda\delta^2\#\mathbb{T}_\lambda.\end{equation} On the other hand, for each $k\geq\log_2\lambda$ we have $$\sum_{Y\in\mathbb{Y}_k} |Y\cap S|<\sum_{T\in\mathbb{T}_\lambda}\sum_{Y\in\mathbb{Y}_T}2^{-k+1}\lambda|Y|\leq 2^{-k+1}\lambda\delta^2\#\mathbb{T}_\lambda.$$ Here the second inequality is because $\#\mathbb{Y}_T\leq(\delta/\eta)^2$. This means that there is a large constant $N$ such that $$\sum_{k\geq N}\sum_{Y\in\mathbb{Y}_k}|Y\cap S|<\frac{c}{2}\lambda\delta^2\#\mathbb{T}_\lambda.$$ Along with equation \ref{1}, we see that there exists some $k\leq N$ such that $$\sum_{Y\in\mathbb{Y}_k}|Y\cap S|>c(\log_2\lambda^{-1}+N)^{-1}\lambda\delta^2\#\mathbb{T}_\lambda.$$ But $\sum_{Y\in\mathbb{Y}_k}|Y\cap S|\leq 2^{-k+1}\lambda|Y|\#\mathbb{Y}_k\leq 2^{-k+1}\lambda\eta^2\cdot(C(E,\epsilon)|S_\eta|\eta^{-2-\epsilon}(2^{-k}\lambda)^{-N})$, from which we conclude $\#\mathbb{T}_\lambda\leq C(E,\epsilon)\delta^{-2-3\epsilon}\lambda^{-N}|S_\eta|$, recalling $\eta=\delta^{3}$ and $C(E,\epsilon)$ changes from line to line. The $\log_2 \lambda^{-1}$ factor is irrelevant as we may assume $\lambda\geq\delta$. This concludes the proof of equation \ref{key} and the proof of the proposition.
\end{proof}

\section{The Polynomial Partitioning Argument}
Once we have the polynomial Wolff axioms, we can use polynomial partitioning to derive the necessary Kakeya bound. The original argument was invented by Guth to make progress in the restriction conjecture, but there have been many improvements since then. In this paper we apply the argument given by Guth and Zahl in \cite{r4}. Since the adaptation is completely straightforward, we merely sketch how their method works without reproducing the proofs of every single lemma.

First, we record some definitions:

\begin{definition}[Grain]
    A grain of complexity $d$ in $\mathbb{R}^3$ is the $C\delta$ neighborhood of a real algebraic variety in $\mathbb{R}^3$ of dimension 2 and complexity $\leq d$.
\end{definition}

Note that each grain is also a semialgebraic set. 

\begin{definition}[Grain decomposition]
Let $Q$ be a collection of disjoint $\delta$ cubes in $\mathbb{R}^3$, a grain decomposition of degree $d$ of error $\epsilon$ is a set of grains $\mathcal{G}$, each of which has complexity $\leq d$, such that the following holds:
\begin{itemize}
    \item For each grain $G\in\mathcal{G}$, there is a subset $Q_G$ of $Q$ such that $\delta$ cube in $Q_G$ is contained in $G$, the $Q_G$'s are disjoint for different $G$. Furthermore we have $\sum_{G\in\mathcal{G}}\# Q_G\geq\delta^{\epsilon}\# Q$ and for each grain $G\in\mathcal{G}$, we have $\delta^{\epsilon}\frac{\#Q}{\#\mathcal{G}}\leq \#Q_G \leq \delta^{-\epsilon}\frac{\#Q}{\#\mathcal{G}}$.
    \item If $T$ is a curved tube, then $\#\{G\in\mathcal{G}, \exists 
 Q\in Q_G\text{ such that } Q\cap T\not=\emptyset\}\leq C_\epsilon\delta^{-\epsilon}\#\mathcal{G}^{\frac{1}{3}}.$
\end{itemize}
\end{definition}

In \cite{r4}, Guth and Zahl proved the existence of a grain decomposition by repeatedly using polynomial partitioning. Their proof was written for the straight case but adapts without change to our tubes, since our tubes are given by polynomials of bounded degree. 

\begin{proposition}[Existence of a grains decomposition\cite{r4}]
    Let $Q$ be a collection of disjoint $\delta$ tubes, then for any $\epsilon>0$, there is a grain decomposition of degree $d(\epsilon)$ and error $\epsilon$.
\end{proposition}

We will also need the following theorem is due to Milnor and Thom; it gives a bound on the number of connected components of a semialgebraic set in terms of its complexity. 
\begin{theorem}[Milnor-Thom Theorem]\label{Milnor} Let $f_i(x_1,...,x_n)$ be a finite collection of real polynomials over $x_1,...,x_n$, let set $S$ be the semialgebraic set given by $f_i\geq 0$ for all $i$, then the number of connected of $S$ is at most $\frac{1}{2}(D+2)(D+1)^{n-1}$, where $D=\sum_i\deg f_i$. 
\end{theorem}

Now, we will use the grains decomposition to prove a version of the multilinear Kakeya inequality:

\begin{lemma}\label{secondone}
Suppose $\mathbb{T}$ is a collection of tubes of the form $T_{x,v,[0,1],\delta}$ such that $\mathbb{T}$ together has $\delta$ separated directions, we will now show that \begin{equation}\int_{B(0,1)}\bigg(\sum_{T_1\in\mathbb{T},T_2\in\mathbb{T}}\chi_{T_1}\chi_{T_2}|v_1\wedge v_2|^{\alpha}\bigg)^{\beta}\lessapprox\delta^{\frac{5}{2}}|\mathbb{T}|^{\frac{3}{2}}. \end{equation}  Here we take $N$ to be the same positive integer as before, $\alpha=\frac{2N}{3N-2}$ and $\beta=\frac{3N-2}{4N-4}$. We also  assume that $\#\mathbb{T}\sim \delta^{-c}$ for some constant $c\in (1,2)$.
\end{lemma}
\begin{proof}
To prove this, first notice that we may ignore all tuples $(T_1,T_2)$ with $|v_1\wedge v_2|<\delta^{100}$; by dyadic pigeonholing, we can then find some $\theta\in[\delta^{100},1)$ such that 
\begin{equation}\int_{B(0,1)}\bigg(\sum_{T_1\in\mathbb{T},T_2\in\mathbb{T}}\chi_{T_1}\chi_{T_2}|v_1\wedge v_2|^{\alpha}\bigg)^{\beta}\lessapprox\int_{B(0,1)}\bigg(\sum_{T_1,T_2\in\mathbb{T},\frac{\theta}{2}<|v_1\wedge v_2|\leq\theta}\chi_{T_1}\chi_{T_2}\,\theta^\alpha \bigg)^\beta.\end{equation}

Now we may replace each tube $T$ in $\mathbb{T}$ with a collection of disjoint $\delta$ cubes that intersect $T$, so that $\chi_T$ now denotes the indicator function of the union of a collection of disjoint $\delta$ cubes; in particular each $\chi_T$ restricted to any $\delta$ cube is either constantly 1 or 0. We use $\mathcal{Q}$ to denote the union of all $\delta$ tubes for all $T\in\mathbb{T}.$ Since there are at most $\delta^{-3}$ disjoint delta cubes, we may ignore $\delta$ cubes $Q$ such that \begin{equation}\int_Q\bigg(\sum_{T_1,T_2\in\mathbb{T},\frac{\theta}{2}<|v_1\wedge v_2|\leq\theta}\chi_{T_1}\chi_{T_2}\,\theta^\alpha\bigg)^\beta\leq\delta^{100}.\end{equation} Using the dyadic pigeonholing principle again, we may take a subcollection $A$ of those delta cubes such that there exists some $\gamma>0$ so that \begin{equation}\label{even}\frac{\gamma}{2}\leq \int_Q\bigg(\sum_{T_1,T_2\in\mathbb{T},\frac{\theta}{2}<|v_1\wedge v_2|\leq\theta}\chi_{T_1}\chi_{T_2}\,\theta^\alpha\bigg)^\beta<\gamma\end{equation} for all $Q\in A$ and \begin{equation}\sum_{Q\in\mathcal{Q}}\int_Q\bigg(\sum_{T_1,T_2\in\mathbb{T},\frac{\theta}{2}<|v_1\wedge v_2|\leq\theta}\chi_{T_1}\chi_{T_2}\,\theta^\alpha\bigg)^\beta\lessapprox \sum_{Q\in A}\int_Q\bigg(\sum_{T_1,T_2\in\mathbb{T},\frac{\theta}{2}<|v_1\wedge v_2|\leq\theta}\chi_{T_1}\chi_{T_2}\,\theta^\alpha\bigg)^\beta.\end{equation}
To make calculation cleaner later, we will pick a $\mu>0$ be so that 
\begin{equation}
\gamma=\mu^{2\beta}\theta^{\alpha\beta}\delta^3
\end{equation}
Therefore, we need to show that for every $\epsilon>0$, we have
\begin{equation}\label{goal}
    \#A\mu^{2\beta}\theta^{\alpha\beta}\delta^{3}\lessapprox\delta^{\frac{5}{2}}\#\mathbb{T}^{\frac{3}{2}}.
\end{equation}
Ultimately the proof of above inequality rests on the following Multilinear Kakeya inequality (in our case this can be proven directly): let $\mathbb{T}$ be a collection of curved tubes with $\delta$ separated directions, then by the "transversality" condition as equation \ref{transversality}, we have
\begin{equation}\label{Multi}
    \int\sum_{T_1,T_2\in\mathbb{T}}\chi_{T_1}\chi_{T_2}|v_1\wedge v_2|\leq\delta^3\#\mathbb{T}^2.
\end{equation}
However, to make efficient use of equation \ref{Multi}, we will consider shorter "subtubes" of $T\in\mathbb{T}$, and obtain a tighter upper bound using the polynomial Wolff axiom than simply $\mathbb{T}$.

Namely, we apply the grains decomposition lemma from above to $A$. For any $\epsilon>0$, this gives a collection of grains $\mathcal{G}$ with error $\epsilon$ and complexities bounded above by $d(\epsilon)$.  We can then rewrite our integral as 
\begin{equation}\int_{\cup_{Q\in A}Q}\bigg(\sum_{\substack{T_1\in\mathbb{T},T_2\in\mathbb{T}\\ \frac{\theta}{2}\leq|v_1\wedge v_2|<\theta}}\chi_{T_1}\chi_{T_2}\,\theta^\alpha\bigg)^\beta\lessapprox\sum_{G\in\mathcal{G}}\int_{\cup_{Q\subset G}Q} \bigg(\sum_{\substack{T_1\in\mathbb{T},T_2\in\mathbb{T}\\ \frac{\theta}{2}\leq|v_1\wedge v_2|<\theta}}\chi_{T_1}\chi_{T_2}\,\theta^\alpha\bigg)^\beta.\end{equation}
Note that on the right hand side of above, we are no longer dealing with full length tubes; rather we are dealing with the intersection of tubes and grains, which in general will be sets of smaller diameter. If $K$ is any set and $x\in K$, we denote $CC(K,x)$ the connected component of $K$ containing $x$, and we use $CC(K)$ to denote all connected components of $K$. Apply the dyadic pigeonholing principle again, we can find $l_1,l_2\geq 0$ such that 
\begin{multline}
\sum_{G\in\mathcal{G}}\int_{\cup_{Q\subset G}Q} \bigg(\sum_{\substack{T_1\in\mathbb{T},T_2\in\mathbb{T}\\ \frac{\theta}{2}\leq|v_1\wedge v_2|<\theta}}\chi_{T_1}\chi_{T_2}\,\theta^\alpha\bigg)^\beta\\ \lessapprox \sum_{G\in\mathcal{G}}\int_{\cup_{Q\subset G}Q}\bigg(\sum_{\substack{T_i\in\mathbb{T},\frac{\theta}{2}\leq|v_1\wedge v_2|<\theta\\ K_i\in CC(T_i\cap G),\frac{l_i}{2}\leq\text{diam}K_i<l_i}}\chi_{T_1}\chi_{T_2}\theta^\alpha\bigg)^\beta\end{multline} 
Here for the inner sum on the right hand side, we are summing over $T_1,T_2$, and over $CC(T_1\cap G)$ and $CC(T_2\cap G).$ Assume $l_1\geq l_2$, we cover $\mathbb{R}^3$ with a collection $\mathcal{B}$ of finitely overlapping radius $100l_1$ balls, such that every subset of $\mathbb{R}^3$ with diameter $<l_1$ is contained in one of the balls. We call a set of the form $B\cap G, B\in\mathcal{B},G\in\mathcal{G}$ a subgrain and denote the collection of subgrains $\mathcal{G}'$. For each subgrain $G'$, we denote $Q_{G'}=\{Q\in A|\,Q\subset G'\}$. We may use the dyadic pigeonhole principle on $\#Q_{G'}$ for each subgrain $G'$; after discarding some subgrains $G'$, we may assume that there exists $t>0$ such that
 \begin{equation}\label{2}
    \frac{t}{2}\leq\#Q_{G'}<t
 \end{equation} for each subgrain $G'$, and 
 \begin{equation}\label{3}
     \sum_{G'\in\mathcal{G}'}\#Q_{G'}\gtrapprox\#A.
 \end{equation}
 By equation \ref{even}, we then have 
\begin{multline}
    \label{4}
    \sum_{G\in\mathcal{G}}\int_{\cup_{Q\subset G} Q}\bigg(\sum_{\substack{T_i\in\mathbb{T},\frac{\theta}{2}\leq|v_1\wedge v_2|<\theta\\ K_i\in CC(T_i\cap G),\frac{l_i}{2}\leq\text{diam}K_i<l_i}}\chi_{T_1}\chi_{T_2}\theta^\alpha\bigg)^\beta\\ \lessapprox\sum_{G'\in\mathcal{G}'}\int_{\cup_{Q\in Q_{G'}}Q}\bigg(\sum_{\substack{T_i\in\mathbb{T},\frac{\theta}{2}\leq|v_1\wedge v_2|<\theta\\ K_i\in CC(T_i\cap G),\frac{l_i}{2}\leq\text{diam}K_i<l_i}}\chi_{T_1}\chi_{T_2}\theta^\alpha\bigg)^\beta
\end{multline}
By assumption we can deduce that for each $T\in\mathbb{T}, G\in\mathcal{G}$ and $K\in CC(T\cap G)$, if $\text{diam } K\in[\frac{l_i}{2},l_i)$ for some $i=1,2$, then $K$ is contained in a subgrain $G'$. For each subgrain $G'\in\mathcal{G}'$ that is the intersection of a grain $G$ and a ball $B\in\mathcal{B}$, we define 
 \begin{equation}
     \mathbb{T}_{i,G'}=\bigg\{T\cap G'|T\in\mathbb{T}, \exists K\in CC(T\cap G)\text{ such that}\text{ diam }K\in\Big[\frac{l_i}{2},l_i\Big),K\subset G'\bigg\}.
 \end{equation}
We now apply the dyadic pigeonhole principle on $\#\mathbb{T}_{i,G'}$; there exist $N_1,N_2>0$ such that after discarding some subgrains, for each $G'\in\mathcal{G'}$ and $i=1,2$ we have
\begin{equation}\label{10}
    \frac{N_i}{2}\leq \#\mathbb{T}_{i,G'}<N_i;
\end{equation}
and equation \ref{3}, and hence equation \ref{4} still hold. Now notice that if $T$ is a curved tube and $G$ is a grain, then $T\cap G$ is a semialgebraic set of bounded complexity, hence it has $O(d(\epsilon))$ connected components by the Milnor-Thom theorem. In particular, the number of elements in $CC(T\cap G)$ with diameter in $[\frac{l_i}{2},l_i)$ is bounded above by $O(d(\epsilon))$. Recall the balls $\mathcal{B}$ are finitely overlapping, this implies that each such $K\in CC(T\cap G)$ is contained in $O(1)$ number of subgrains. Hence a double counting argument and the fact that $\mathcal{G}$ is a grain decomposition imply that
\begin{equation}
    N_i\#\mathcal{G}'\lessapprox\#\mathbb{T}\#\mathcal{G}^{\frac{1}{3}},
\end{equation}
hence 
\begin{equation}
    N_i\lessapprox\#\mathbb{T}\#\mathcal{G}'^{-\frac{2}{3}}.
\end{equation}
We now do one more round of dyadic pigeonholing. After discarding some $\delta$ cubes, we may assume that there exist $\mu_1,\mu_2>0$ such that if $Q\in A$ is a $\delta$ cube contained in some subgrain $G'$, and for each $i=1,2$, we have
\begin{equation}
    \frac{\mu_i}{2}<\#\{T\in\mathbb{T}_{i,G'}|\,Q\subset T\}<\mu_i;
\end{equation}
furthermore equation \ref{3} and hence equation \ref{4} still hold. Recall each indicator function $\chi_T$ is either 1 or 0 on a $\delta$ cube in $A$. Equation \ref{even} implies that 
\begin{equation}\label{muu}
\frac{\mu^2}{2}\leq\sum_{T_1,T_2\in\mathbb{T},\frac{\theta}{2}\leq|v_1\wedge v_2|<\theta}\chi_{T_1}(x)\chi_{T_2}(x)<\mu^2
\end{equation} whenever $x\in Q$ for some $Q\in A$. Therefore, we can conclude
\begin{equation}
    \mu\lesssim(\mu_1\mu_2)^{\frac{1}{2}}\leq\max(\mu_1,\mu_2).
\end{equation}
In choosing $\mu_1,\mu_2$ we discarded some $\delta$ cubes contained in each $\mathcal{G}'$, we now do another round of dyadic pigeonholing by discarding some subgrains in $\mathcal{G}'$ to restore equation \ref{2} while maintaining equations \ref{3} and \ref{4}.

We will need one more ingredients. Firstly, by Wongkew's lemma, for each subgrain $G'$, we have $|G'|\lesssim_{\epsilon}l_1^2\delta$. This implies that 
\begin{equation}\label{fromw}
    N_i\lessapprox l_1^2\delta\delta^{-2}l_i^{-N}=l_1^2\delta^{-1}l_i^{-N}.
\end{equation}
Now, we denote $A'$ to be the collection of $\delta$ cubes in $A$ that are contained in some subgrain $G'\in\mathcal{G}'$. Then as a result of all the dyadic pigeonholing above, we have 
\begin{equation}\label{22}
    \text{RHS of }\ref{4}\approx\#A'\mu^{2\beta}\theta^{\alpha\beta}\delta^3. 
\end{equation}

Moreover, for each $i=1,2$ we have 
\begin{align}
    \mu_i\# A'\delta^3 &\lesssim \sum_{G'\in\mathcal{G'}}\sum_{T\in\mathbb{T}_{i,G'}}|T\cap G'| \\
    &\lesssim \#\mathcal{G}'N_i l_i\delta^2\\
    &\lesssim \#\mathcal{G}'^{\frac{1}{3}}\#\mathbb{T}l_i\delta^2
\end{align}

In particular,  since we agreed that $l_1\geq l_2$, this gives

\begin{equation} \label{help}
    l_1\gtrsim \max(\mu_1,\mu_2)\#A'\delta(\#\mathcal{G}')^{-\frac{1}{3}}\#\mathbb{T}^{-1}
\end{equation}

Together with equation \ref{fromw}, we have 
\begin{equation}\label{final}
N_1\lesssim \delta^{-1}l_1^{-(N-2)}\lesssim \max(\mu_1,\mu_2)^{-(N-2)}\#A'^{-(N-2)}\delta^{-(N-1)}\#\mathcal{G}'^{\frac{N-2}{3}}\#\mathbb{T}^{N-2}.
\end{equation}

Lastly, for each subgrain $G'$, we apply equation \ref{Multi} to $\mathbb{T}_{1,G'}$ and $\mathbb{T}_{2,G'}$; this gives us

\begin{equation}
    \int\sum_{T_1\in\mathbb{T}_{1,G'}, T_2\in\mathbb{T}_{2,G'}}\chi_{T_1}\chi_{T_2}|v_1\wedge v_2|\leq\delta^3\prod_i\#\mathbb{T}_{i,G'}.
\end{equation}
This implies by equation \ref{muu}
\begin{equation}\label{28}
    \frac{\#A'}{\#\mathcal{G}'}\mu^2\delta^3\theta\lesssim\delta^3(N_1N_2). 
\end{equation}
We can now collect all the inequalities we have derived.  By equation \ref{22} and equation \ref{goal}, the inequality we are after is: 
\begin{equation}\label{target}
\#A'\mu^{2\beta}\theta^{\alpha\beta}\delta^3\lessapprox\delta^{\frac{5}{2}}\#\mathbb{T}^{\frac{3}{2}}.
\end{equation}
This follows from 
\begin{align}
    \#A'(\#\mathcal{G}')^{-1}\theta\delta^3\mu^2 &\lesssim\delta^3(N_1N_2)\;\;\;\;\;\;\;\;\;\;\;\;\;\;\;\;\;\;\;\text{by }\ref{28}\\
    &\lesssim \delta^3\big(\#\mathbb{T}\#\mathcal{G}'^{-\frac{2}{3}}\big)^{1-\frac{1}{N}} \allowbreak \big(\max(\mu_1,\mu_2)^{-(N-2)}\#A'^{-(N-2)}\delta^{-(N-1)}\#\mathcal{G}'^{\frac{N-2}{3}}\#\mathbb{T}^{N-2}\big)^{\frac{1}{N}} \\ &\;\;\;\;\;\;\;\;\;\;\;\cdot \big(\#\mathbb{T}\#\mathcal{G}'^{-\frac{2}{3}}\big)   \\
    &\lesssim \delta^{3-\frac{N-1}{N}}\#\mathbb{T}^{3-\frac{3}{N}}\allowbreak\mu^{-(1-\frac{2}{N})}\# A'^{-(1-\frac{2}{N})}\mathcal{G}'^{-1}
\end{align}
where for the second line we used equation \ref{final} and for the last line we used the fact that $\mu\lesssim\max(\mu_1,\mu_2)$. This gives
\begin{equation}
    \#A'^{2-\frac{2}{N}}\mu^{3-\frac{2}{N}}\theta\lessapprox\delta^{-(1-\frac{1}{N})}\mathbb{T}^{3-\frac{3}{N}}
\end{equation}
 and finishes the proof, recalling the values of $\alpha$ and $\beta$. 
 \end{proof}

\section{From Multilinear To Linear}
To go from multilinear Kakeya inequality to linear Kakeya inequality, we can simply use the broad narrow argument by Bourgain and Guth in \cite{Bourgain2011}. The end result is the following
\begin{proposition}\label{thirdone}
    Let $\mathbb{T}$ be a collection of tubes of the form $T_{x,v,[0,1],\delta}$ with $\sim\delta$ separated directions $v$'s in $[0,1]^2$ and such that $\#\mathbb{T}\gtrapprox\delta^{-s}$ for some $s\in(1,2]$, then there exists $C>0$ such that \begin{equation}\label{linear}\big|\big|\sum_{T\in\mathbb{T} }\chi\big|\big|_p \leq C\delta^{-2+\frac{4\beta-\frac{1}{2}}{p}}
\end{equation}
whenever $p\geq 2\beta.$ Here and in the proof later $\alpha=\frac{2N}{3N-2}$ and $\beta=\frac{3N-2}{4N-4}$ as previously defined. Here we make the constant $C>0$ explicit so the proof is more easily understandable. 
\end{proposition}
\begin{proof}

To upper bound $||\sum_{T\in\mathbb{T}}\chi_T||_{L^p}$. We divide $[0,1]^2$ into finitely overlapping caps of radius $\rho\geq\delta$ and use $\tau_\rho$ to denote the collection of tubes in $\mathbb{T}$ with directions lying in a cap; we instead upper bound $||\sum_{T\in\tau_\rho}\chi_T||_{L^p}$ for each cap $\tau_\rho$. Here we use $T\in\tau_\rho$ to denote that the direction of $T$ lies in $\tau_\rho$. We need to show the stronger result:

$$\big|\big|\sum_{T\in\tau_\rho}\chi_T\big|\big|_{L^p}\leq C\delta^{-2+\frac{4\beta-\frac{1}{2}}{p}}\rho^{2+\frac{3-\alpha\beta-4\beta}{p}}.$$

We first notice note that if $\rho=\delta$, then $LHS=\delta^{\frac{2}{p}}$ and $RHS=\delta^{\frac{\frac{5}{2}-\alpha\beta}{p}}$, and the result holds trivially in this case since $\alpha\beta=\frac{2N}{4N-4}>\frac{1}{2}$. This will serve as the base case for the induction. Next, we do induction on $\rho$ in the same way as done by Bourgain and Guth in dealing with the curved Kakeya problem. To do this we break $\tau_\rho$ into smaller caps $\tau_{\rho/K}$ for some large constant $K$ to be chosen later. A point $x\in\mathbb{R}^3$ will be called narrow if there are $\leq 10^4$ caps accounting for more than half of the tubes containing $x$. That is, there is a subset $C(x)$ of $\tau_{\rho/K}$'s with cardinality $\leq 10^4$ such that $$\sum_{\tau_{\rho/K}\in C(x)}|\{T: x\in T, v(T)\in\tau_{\rho/K}\}|\geq\frac{1}{2}|\{T:x\in T\}|.$$ A point $x\in\mathbb{R}^3$ will be called broad otherwise. 

The induction hypothesis directly controls the contribution of the $L^p$ norm of $\sum_{T\in\tau_\rho}\chi_T$
on the narrow subset: if $x$ is narrow, then by definition we have $$\sum_{T\in\tau_\rho}\chi(x)\leq 2\sum_{\tau_{\rho/K}\in C(x)}\sum_{T\in\tau_{\rho/K}}\chi_T(x). $$

Holder's inequality and the fact that $x$ is narrow then imply that 
$$\sum_{T\in\tau_\rho}\chi(x)\leq 2\cdot 10^4\bigg(\sum_{\tau_{\rho/K}\in C(x)}\Big(\sum_{T\in\tau_{\rho/K}}\chi_T(x)\Big)^p\bigg)^{\frac{1}{p}}.$$

Integrate, this gives 

\begin{align*}
\int_{\text{narrow}}\Big(\sum_{T\in\tau_\rho}\chi(x)\Big)^p&\leq (2\cdot 10^4)^p\int_{\text{narrow}}\sum_{\tau_{\rho/K}\in C(x)}\Big(\sum_{T\in\tau_{\rho/K}}\chi_T(x)\Big)^p\\
&\leq (2\cdot 10^4)^p\cdot C^p\sum_{\text{all the }\tau_{\rho/K}}\delta^{-2p+4\beta-\frac{1}{2}}\bigg(\frac{\rho}{K}\bigg)^{2p+3-\alpha\beta-4\beta}\\&\leq (2\cdot 10^4)^p\cdot C^p\cdot 100\cdot K^{\alpha\beta+4\beta-1-2p}\delta^{-2p+4\beta-\frac{1}{2}}\rho^{2p+3-\alpha\beta-4\beta}.
\end{align*}
Here we use the induction hypothesis to derive the second inequality, and the last inequality is because there are at most $100K^2$ caps $\tau_{\rho/K}$. We let $K>0$ be a constant large enough; the induction closes when $x$ is narrow because we assumed $p>2\beta$ and a simple calculation gives $\alpha\beta+4\beta-1-2p<0$. If we take $K>0$ to be a sufficiently large constant, then we can assume that 

$$\int_{\text{narrow}}\Big(\sum_{T\in\tau_\rho}\chi(x)\Big)^p\leq\frac{C^p}{2}\delta^{-2p+4\beta-\frac{1}{2}}\rho^{2p+3-\alpha\beta-4\beta}$$

Next, to deal with the $L_p$ norm for broad points, we need to the following observation: if we have a point contained in the intersection of two curved tubes $T_{x,v,[0,1],\delta}$ and $T_{x',v',[0,1],\delta}$, the point can be expressed as $\Phi(x,v,t)+\epsilon$ and $\Phi(x',v',t)+\epsilon'$, where $\epsilon,\epsilon'$ are vectors of length $\leq\delta$. Remember the notation $b_1(t)=\begin{bmatrix}
    \beta_1(t) \\
    \gamma_1(t)
\end{bmatrix}$ and $b_2(t)=\begin{bmatrix}
    \beta_2(t) \\
    \gamma_2(t)
\end{bmatrix}$ we have 
\begin{multline}\label{transversality}
    \bigg|\frac{d}{dt}(t,\Phi(x,v,t))\wedge\frac{d}{dt}(t,\Phi(x',v',t))\bigg|=\Bigg|\Bigg|\begin{bmatrix}
    1 \\ (x,v)\cdot b_1'(t) \\ (x,v)\cdot b_2'(t)
\end{bmatrix}\wedge \begin{bmatrix}
    1 \\ (x',v')\cdot b_1'(t) \\ (x',v')\cdot b_2'(t)
\end{bmatrix}\Bigg|\Bigg| \\ \geq\Bigg|\Bigg|
    \begin{bmatrix}
        b_1'(t)\\
        b_2'(t)
    \end{bmatrix}\cdot
    \begin{bmatrix}
        x_1-x_1'\\
        x_2-x_2'\\
        v_1-v_1'\\
        v_2-v_2'
    \end{bmatrix}
    \bigg|\bigg|
\end{multline}    
    However, we already know that $|\Phi(x,v,t)-\Phi(x',v',t)|\leq 2\delta$, which means that $$\Bigg|\Bigg|
    \begin{bmatrix}
        b_1(t)\\
        b_2(t)
    \end{bmatrix}\cdot
    \begin{bmatrix}
        x_1-x_1'\\
        x_2-x_2'\\
        v_1-v_1'\\
        v_2-v_2'
    \end{bmatrix}
    \bigg|\bigg|\leq 2\delta. $$ However, lemma \ref{reduction} implies that the matrix $\begin{bmatrix}
        b_1(t)\\
        b_2(t)\\
        b_1'(t)\\
        b_2'(t)
    \end{bmatrix}$ is invertible for all $t$, so if we assume $v$ and $v'$ are $C\delta$ separated for a constant $C>0$ large enough, then we must have $$\Bigg|\Bigg|
    \begin{bmatrix}
        b_1'(t)\\
        b_2'(t)
    \end{bmatrix}\cdot
    \begin{bmatrix}
        x_1-x_1'\\
        x_2-x_2'\\
        v_1-v_1'\\
        v_2-v_2'
    \end{bmatrix}
    \bigg|\bigg|\geq c||(x,v)-(x',v')||$$ for $c>0$ a small constant. This gives the transversality needed for the "broad-narrow" argument.
    
    If $x$ is a broad point, by definition of a broad point and the observation above, we see that for most tuples of tubes in $C(x)$, we have  $v_1(x)\wedge v_2(x)\gtrsim \frac{\rho}{K}$, where $v_1,v_2$ are the unit tangent vectors to the central curves at $x$. Therefore, we have $$\bigg|\sum_{T\in\tau_\rho}\chi_T(x)\bigg|^2\lesssim K^\alpha\rho^{-\alpha}\sum_{T_1,T_2\in\tau_\rho}\chi_{T_1}\chi_{T_2}|v_1(x)\wedge v_2(x)|^\alpha$$ if $x$ is a broad point. We then deduce that \begin{align*}\int_{\text{Broad}}\bigg|\sum_{T\in\tau_\rho}\chi_T(x)\bigg|^{2\beta}&\lesssim\int_{\text{Broad}} K^{\alpha\beta}\rho^{-\alpha\beta}\bigg(\sum_{T_1,T_2\in\tau_\rho}\chi_{T_1}\chi_{T_2}|v_1(x)\wedge v_2(x)|^{\alpha}\bigg)^{\beta}\\&\lesssim K^{\alpha\beta}\rho^{-\alpha\beta}\delta^{\frac{5}{2}}\rho^3\delta^{-3}=K^{\alpha\beta}\rho^{3-\alpha\beta}\delta^{-\frac{1}{2}}\end{align*} where the second inequality comes from the multilinear bound. Furthermore, we have the trivial bound $$\bigg|\bigg|\sum_{T\in\tau_\rho}\chi_T(x)\bigg|\bigg|_{L^\infty}\lesssim(\rho\delta^{-1})^2.$$
Interpolating, we get $$\int_{\text{Broad}}\bigg|\sum_{T\in\tau_\rho}\chi_T(x)\bigg|^{p}\lesssim K^{\alpha\beta}\rho^{2p+3-\alpha\beta-4\beta}\delta^{-2p+4\beta-\frac{1}{2}}.$$
Since $K$ has been chosen to be a constant independent of $C$, if we have chosen $C>0$ to be a sufficiently large constant, then we have 
$$\int_{\text{broad}}\Big(\sum_{T\in\tau_\rho}\chi(x)\Big)^p\leq\frac{C^p}{2}\delta^{-2p+4\beta-\frac{1}{2}}\rho^{2p+3-\alpha\beta-4\beta}.$$
Together the narrow and broad estimates close the induction. This concludes the argument of Bourgain and Guth. 
\end{proof}

\section{Finishing the Argument}

Our last step is converting the $L^p$ Kakeya bound into a restricted projection bound. The argument uses a discretized version of the Frostman's lemma, similar arguments are very common in the literature, see for example section 2.1 of  \cite{plane} and section 1.2 of \cite{zahl}. The discretization procedure we use is a combination of both approaches. First we need a discrete version of Frostman's lemma; a proof is contained in lemma 3.13 of \cite{fassler}:

\begin{definition}
    Let $\delta,s>0$ and $P\subset\mathbb{R}^n$ be a finite $\delta>0$ separated set, we say $P$ is a $(\delta,s)$ separated set for every ball $B(x,r)$, we have $\#(P\cap B(x,r))\lesssim\big(\frac{r}{\delta}\big)^s$.
\end{definition}

\begin{lemma}[Frostman's Lemma, Discrete]\label{Frostman}
Let $\delta,s>0$ and $B\subset\mathbb{R}^n$ be a Borel measurable subset with $a:=H^s_\infty(B)>0$, then there is a $(\delta,s)$ subset $P\subset B$ such that $\#P\gtrsim a\delta^{-s}.$ Here $H^s_\infty$ denotes the Hausdorff content.
\end{lemma}

We will also need the Frostman's lemma for analytic sets.

\begin{lemma}[Frostman's Lemma]
Let $B\subset\mathbb{R}^n$ be analytic and $s>0$, then $H^s(B)>0$ if and only if there is a probability measure $\mu$ supported on $B$ such that for every ball $B(x,r)\subset\mathbb{R}^n$, we have $\mu(B(x,r))\leq r^s$. Here $H^s$ denotes the Hausdorff measure of dimension $s$. We call $\mu$ a Frostman's measure of dimension $s$.
\end{lemma}

Frostman's lemma can be directly proved for closed sets, and a proof is presented in chapter 3 of \cite{Fractal}, where Frostman's lemma for analytic sets in general is also stated. A proof of the lemma for Borel sets is in Theorem B.2.5 of \cite{Fractal}, where it is shown that if $A\subset\mathbb{R}^n$ is Borel and $H^s(A)>0$, then $A$ contains a compact subset $K$ with $H^s(K)>0$.
 
Lastly, we will make use of two elementary measure theory facts:

\begin{lemma}\label{basic}
\text{ }\\
1)  If $A\subset\mathbb{R}^n$ is Borel measurable, $s\geq 0$ and $H^s_\infty(A)=0$, then for any $\epsilon>0$, we can cover $A$ with dyadic cubes $D\in\mathcal{D}$ with side length $r(D)<\epsilon$ such that $\sum_{D\in\mathcal{D}}r(D)^s<\epsilon$ and the interiors of the dyadic cubes are disjoint.\\
2) Let $B\subset\mathbb{R}$ be a Borel measurable set with positive Lebesgue measure and $\delta>0$. Then any maximally $\delta$ separated subset $B_\delta$ of $B$ has cardinality $\gtrsim|B|\delta^{-1}.$ 
\end{lemma}
\begin{proof} For 1), clearly we can assume $s>0$ and we only need to make sure $\sum_{D\in\mathcal{D}}r(D)^{s}<\epsilon$. Then 1) follows from the definition of Hausdorff content and the observations that any set of diameter bounded by some $\delta>0$ is contained in a dyadic cube of side length at most $100\delta$ and if $A,B$ are two dyadic cubes (of maybe different side lengths), then either $A$ contains $B$, or $B$ contains $A$, or the interiors of $A$ and $B$ are disjoint.

For 2), if $\#B_\delta<\frac{1}{100}|B|\delta^{-1}$, by maximality $B$ is contained in the $\delta$ neighborhood $B_\delta$ of $B$, so $|B|\leq 2\delta\cdot\#B_\delta<\frac{1}{50}|B|$, contradiction.
\end{proof}

Equipped with these lemmas, we now move to the restricted projection result theorem \ref{Main} we are after in this paper. We restate the result for convenience with explicit $\epsilon$. 

\begin{theorem}
Suppose $\Pi_t, t\in[0,1]$ is a family of two dimensional subspaces of $\mathbb{R}^4$ and let $P_t:\mathbb{R}^4\rightarrow\Pi_t$ be the orthogonal projection onto $\Pi_t$. Furthermore assume that we have $b_1(t), b_2(t): [0,1]\rightarrow\mathbb{R}^4$ such that $\Pi_t=\text{span }(b_1(t),b_2(t))$ and each coordinate function of $b_i$ is a polynomial in $t$. Furthermore, assume that for any subspace $\pi$ of $\mathbb{R}^4$, $\dim P_t(\pi)=\min(\dim\pi,2)$ for all but finitely many $t\in[0,1]$. Then $\dim_H(P_t(A))\geq 1+\frac{1}{N}$ for almost all $t$ whenever $A$ is a Borel subset of $\mathbb{R}^4$ and $\dim_H A=2$. Here, $N$ is a large integer depending only on the degrees of coordinate functions of $\beta_i$'s and $\gamma_i$'s. 
\end{theorem}

\begin{proof}
Without loss of generality we may assume $A\subset [0,1]^4$. We can also redefine $P_t$ to be the map $(x,v)\mapsto\Phi(x,v,t)$. Let $s\in(0,2)$ be arbitrarily close to 2. We can pick $s'\in(s,2)$. For any positive $a<1+\frac{1}{N}$, we let $\Theta=\Big\{t\in[0,1]\Big|\,\dim_H(P_t(A))<a\Big\}$ be the exception set and assume for contradiction that $|\Theta|>0$.  Let $\nu$ be the Lebesgue outer measure on $[0,1]$ restricted to $\Theta$ since we do not know if $\Theta$ is Lebesgue measurable. Fix $\epsilon>0$. Since by definition $\dim_H A=\inf\{w\geq 0, H^w_\infty(A)=0\}$, we have $H^{a}_\infty(P_t(A))=0$ for each $t\in\Theta$. For each $t\in\Theta$ we can cover $P_t(A)$ with collections of disjoint dyadic cubes $\mathcal{D}_t$, each of which has side length $2^{-j}$ for some positive integer $j$ so that $2^{-j}<\epsilon$. Furthermore, we can make sure that $\sum_{D\in\mathcal{D}_t}r(D)^a<\epsilon$. For each positive integer $k$, we let $\mathcal{D}_{k,t}$ denote the set of dyadic cubes of side length $2^{-k}$ in $\mathcal{D}_t$. Then we must have $\#\mathcal{D}_{k,t}\leq \epsilon 2^{ak}$ for each $k\in\mathbb{Z}^+$.

Let $\nu_A$ be a Frostman's measure of dimension $s'$ supported on $A$. For each $t\in\Theta$, we let $\mathcal{T}_t=\{P_t^{-1}(D),D\in\mathcal{D}_t\}$ and $\mathcal{T}_{k,t}=\{P_t^{-1}(D),D\in\mathcal{D}_{k,t}\}$ where $k\in\mathbb{Z}^+$ and $2^{-k}<\epsilon$, then $\mathcal{T}_t=\bigcup_{k\in\mathbb{Z}^+}\mathcal{T}_{k,t}$ and $A\subset\bigcup_{T\in\mathcal{T}_t}T$ for each $t\in\Theta$. Therefore, for each $t\in\Theta$, there exists a $k(t)\in\mathbb{Z}^+$ such that $$\nu_A\big( A\bigcap\cup_{T\in T_{k(t)}}T\big)\geq\frac{1}{10k(t)^2}\nu_A(A)=\frac{1}{10 k(t)^2}.$$
For each positive integer $k$ we define $\Theta_k=\{k(t)=k|t\in\Theta\}$, then again we can conclude that there exists a fixed $k$ such that the Lebesgue outer measure $\nu(\Theta_k)>0$. We fix this $k$ and set $\delta=2^{-k}$, then $\delta<\epsilon$ by our assumptions. Furthermore, we let $\Theta'$ be a maximum $\delta$ separated subset of $\Theta_k$, then $\#\Theta'\gtrsim\delta^{-1}$ by 2 of lemma \ref{basic}.

Next, we consider the set $S$ defined as $\{(z,t)\in A\times\Theta': z\in\cup_{T\in \mathcal{T}_{k,t}}T\}$, and we use $\mu$ to denote the counting measure on $\Theta'$. Furthermore, define sections as $S_z=\{t\in\Theta'|(z,t)\in S\}$\ and $S_t=\{z\in A|(z,t)\in S\}.$ Then by definition of $\Theta'$ we have $$(\nu_A\times\mu)(S)\geq\frac{1}{10k^2}\mu(\Theta'),$$ which implies that  $$(\nu_A\times\mu)\Big(\big\{(z,t)\in S:\mu(S_z)\geq\frac{1}{20k^2}\mu(\Theta')\big\}\Big)\geq\frac{1}{20k^2}\mu(\Theta');$$ which further implies that $$\nu_A\bigg(\big\{z:\mu(S_z)\geq\frac{1}{20k^2}\mu(\Theta')\big\}\bigg)\geq\frac{1}{20k^2}.$$ 
This implies that $\dim_H \{z:\mu(S_z)\geq\frac{1}{20k^2}\mu(\Theta')\big\}>s'$ as $\nu_A$ is a Frostman measure. By the Kauffman's version of Marstrand's theorem, we can find a two dimensional subspace $\Pi\subset\mathbb{R}^4$, such that $\dim_H P(\{z:\mu(S_z)\geq\frac{1}{20k^2}\mu(\Theta')\big\})\geq s'$, where $P:\mathbb{R}^4\mapsto\Pi$ is the orthogonal projection onto $\Pi$. We can then apply lemma \ref{Frostman} to the image $P(\{z:\mu(S_z)\geq\frac{1}{20k^2}\mu(\Theta')\big\})$. This gives a $\delta$ separated subset $A'\subset\big\{z\in A:\mu(S_z)\geq\frac{1}{20k^2}\mu(\Theta')\big\}$ with $\#A'\gtrapprox\delta^{-s}$, and such that $P$ is a bijection on $A'$ and $P(A')$ is a $\delta$ separated subset of $\Pi$. 

Unpacking the definitions, we see that for each $z\in A'$ and $t\in S_z$, $P_t(z)$ is contained in a side length $\delta$ square in $\mathcal{D}_{k,t}$. Furthermore a straightforward calculation shows that there exists a sufficiently large $C>0$ such that $|\Phi(x,v,t)-\Phi(x',v',t)|\leq \delta$ if $|t'-t|\leq\frac{\delta}{C}$. For each $t\in\Theta'$, let $E_{k,t}\subset\mathbb{R}^3$ be the three dimensional cube given by the product of $[t-\frac{\delta}{C},t+\frac{\delta}{C}]$ and the squares in $\mathcal{D}_{k,t}.$ Since points in $S_z$ are $\delta$ separated, we see that if we let $\Gamma_z:t\rightarrow(P_t(z ),t)$ then the $C\delta$ neighborhood $\Gamma_z^{C\delta}$ of $\Gamma_z$ contains $\#\mu(S_z)$ number of sidelength $\approx\delta$ 3-dimensional cubes. Let $E=\cup_{t\in\Theta'}E_{k,t}$. All such cubes for different $z\in A'$ are disjoint with each other because $\Theta'$ is $\delta$ separated. Therefore we are able to conclude that  
$$\int_E \sum_{z\in A'}1_{\Gamma_z^{C\delta}}(t,y)\,dy\,dt=\sum_{z\in A'}\int_E 1_{\Gamma_z^{C\delta}}(t,y)\,dy\,dt\geq\# A'\cdot\text{ lower bound for the cardinality of }S_z\cdot\,\delta^3$$$$\gtrapprox\delta^{-s}\cdot \delta^{3}\cdot\delta^{-1}=\delta^{2-s}.$$
On the other hand we have $|E|\lesssim \delta^3\cdot\#\Theta'\cdot \text{ upper bound for }\#\mathcal{D}_{k,t}\lessapprox\delta^{3}\times\delta^{-1}\times\epsilon\delta^{-a}\leq\delta^{2-a}$, Holder's inequality then implies that $$\bigg|\bigg|\sum_{z\in A'}1_{\Gamma_z^{C\delta}}(t,y)\bigg|\bigg|_{L^{2\beta}}\gtrapprox\delta^{2-s}\delta^{-(2-a)(1-\frac{1}{2\beta})}.$$

Notice that if we write $z\in A'$ as $z=(x,v)$ where $x,v\in\mathbb{R}^2$, then $\Gamma_z(t)=(\Phi(x,v,t),t)$. The collection of $z\in A'$ might not have $\delta$ separated directions $v$, but by construction $P(A')$ is $\delta$ separated. Therefore we may apply proposition \ref{linear} with $p=2\beta$, and compare with the inequality above, we see that $\delta^{-2+\frac{4\beta-\frac{1}{2}}{2\beta}}\gtrapprox \delta^{2-s-(2-a)(1-\frac{1}{2\beta})}$. Now remember by definition $\delta<\epsilon$, so by letting $\epsilon\downarrow 0$ and $s\uparrow 2$, we have $\frac{1}{4\beta}\geq (2-a)(1-\frac{1}{2\beta})$, which simplies to $a\geq 1+\frac{1}{N}$. This is impossible since we assumed $a<1+\frac{1}{N}$. This shows that $\Big\{t\in[0,1]\Big|\,P_t(A)<a\Big\}$ has Lebesgue outer measure 0 whenever $a<1+\frac{1}{N}$, taking a rational sequence that increases to $1+\frac{1}{N}$, we see that $\dim_H P_t(A)\geq 1+\frac{1}{N}$ for almost every $t\in[0,1].$
\end{proof}
This concludes the proof. 

\bibliographystyle{ieeetr} 
\bibliography{references}

\end{document}